\documentclass[12pt]{article}
\usepackage{pgf,tikz}
\usepackage{calc,graphicx, complexity}
\usepackage{graphics}
\everymath{\displaystyle}
\usepackage{graphicx}
\usepackage{color}
\usepackage{amssymb}
\usepackage{amsmath}

\newtheorem{prethm}{{\bf Theorem}}

\newenvironment{thm}{\begin{prethm}{\hspace{-0.5
				em}{\bf .}}}{\end{prethm}}
\newtheorem{prelemma}{{\bf Lemma}}

\newtheorem{preex}{{\bf Example}}

\newtheorem{preprop}{{\bf Proposition}}

\newenvironment{prop}{\begin{preprop}{\hspace{-0.5em}{\bf .}}}{\end{preprop}}
\newtheorem{precor}{{\bf Corollary}}

\newenvironment{cor}{\begin{precor}{\hspace{-0.5
				em}{\bf .}}}{\end{precor}}
\newtheorem{preremark}{{\bf Remark}}

\newtheorem{preprob}{{\bf Problem}}

\newtheorem{predefin}{{\bf Definition}}

\newenvironment{defin}{\begin{predefin}{\hspace{-0.5
				em}{\bf .}}}{\end{predefin}}
\newtheorem{preconj}{{\bf Conjecture}}

\newtheorem{preprobb}{{\bf Problem}}

\newtheorem{prelem}{{\bf Theorem}}

\newtheorem{precla}{{\bf Claim}}

\newenvironment{proof}{{\bf Proof.}\rm }{\hfill{$\Box$}}

\newtheorem{presolution}{{\bf Solution.}}

\def\newpic#1{}
\def\qed{\ifhmode\unskip\nobreak\fi\quad\ifmmode\Box\else$\Box$\fi}

\title{\vspace{0cm}\Large\bf\noindent Improved bounds on the b-chromatic number using the independence and chromatic numbers}
\author{\large\bf Manouchehr Zaker\footnote{mzaker@iasbs.ac.ir}
\vspace{5mm}\\
Department of Mathematics,\\
Institute for Advanced Studies in Basic Sciences,\\
Zanjan 45137-66731, Iran\\
}

\date{}

\begin{document}
\maketitle
\vspace{-1cm}
\begin{abstract}
\noindent Finding a minimum proper vertex coloring is $\NP$-hard and hard to approximate. $B$-coloring is used as a heuristic approach to graph coloring. A b-coloring of a graph $G$ is a proper vertex coloring where each color class contains at least one vertex (a b-vertex) adjacent to a vertex in every other color class. The maximum number of colors in such a coloring is the b-chromatic number, ${\rm b}(G)$. A ${\rm b}^{\ast}$-coloring is a variation in which a b-vertex is adjacent to a b-vertex in every other color class. We employ the ${\rm b}^{\ast}$-coloring to prove that any $n$-vertex graph $G$ with independence number at most $t$ satisfies ${\rm b}(G) \leq [(t-1)n+t\chi(G)]/(2t-1)$. This bound extends the bounds of Kouider and Zaker (2006) and Alkhateeb and Kohl (2011) and significantly improves the bound in terms of the clique partition number. We show that this bound is sharp for all $t\geq 2$ and $\chi(G)\geq 3$. Furthermore, we provide a refined bound based on the maximum number of vertex-disjoint independent sets of size $t$. Finally, we prove ${\rm b}^{\ast}(G) \leq [(t-2)n+(t-1)\chi(G)]/(2t-3)$ for all $K_{1,t}$-free graphs $G$, a significant improvement over the analogous bound for ${\rm b}(G)$.
\end{abstract}



\section{Introduction}

\noindent All graphs in this paper are undirected without any loops and multiple edges. We denote the maximum degree of a graph $G$ by $\Delta(G)$. Complete graph on $n$ vertices is denoted by $K_n$. For each positive integer $k$ we denote $\{1, \ldots, k\}$ by $[k]$. The complement of a graph $G$ is denoted by $\overline{G}$. For a subset $S$ of vertices in $G$, by $G[S]$ we mean the subgraph of $G$ induced on the elements of $S$. 
A subset of vertices $S$ in a graph $G$ is called independent if the vertices of $S$ do not induce any edge. The maximum size of an independent set, denoted by $\alpha(G)$ is called the independence number of $G$. The maximum number of mutually adjacent vertices is denoted by $\omega(G)$. A proper vertex coloring of a graph $G$ is an assignment of colors $c:V(G)\rightarrow \mathbb{N}$ such that no two adjacent vertices receive same colors. By a color class we mean a subset of vertices having a same color. The chromatic number $\chi(G)$, is the smallest number of colors used in a proper vertex coloring of $G$. We refer to \cite{BM} for the concepts not defined here. In a proper vertex coloring $c$ of $G$, a vertex $u$ is called b-vertex if $u$ has a neighbor of color $j$ for each color $j\not= c(u)$. A proper coloring $c$ is b-coloring if each color class contains a b-vertex. Determining $\chi(G)$ is $\NP$-hard \cite{BM}. $B$-coloring is used as a heuristic approach to graph coloring. The maximum number of colors in a b-coloring of $G$, denoted by ${\rm b}(G)$, is called the b-chromatic number. This quantity is also denoted by $\chi_b(G)$ and $\varphi(G)$. Clearly, ${\rm b}(G)\leq \Delta(G)+1$. It is well-known that there exists a polynomial time algorithm which given a graph $G$, the algorithm provides a vertex coloring of $G$ using at most ${\rm b}(G)$ colors. The literature is full of papers concerning the b-coloring of graphs e.g. the survey paper \cite{JP} and \cite{BR,CLS,HK,IM,KTV,KZ,VBK}. Also to determine ${\rm b}(G)$ is $\NP$-complete for complement of bipartite graphs \cite{BSSV}, for bipartite graphs \cite{KTV} and hard to approximate in general \cite{GK}. Computational complexity of the b-chromatic number problem restricted to $H$-free graphs has been recently investigated in \cite{AELMMP}. By a Grundy-coloring (First-Fit coloring) of $G$ we mean a coloring using say $k$ colors such that for each $i,j\in \{1, 2, \ldots, k\}$ with $i<j$, each vertex of color $j$ in $G$ has a neighbor of color $i$. The maximum $k$ satisfying this property is called the Grundy number (or First-Fit chromatic number) of $G$ and denoted by $\Gamma(G)$ also by $\chi_{_{\sf FF}}(G)$ \cite{Z0}. Many researches are devoted to simultaneous or comparative studies of the Grundy and b-chromatic numbers e.g. \cite{HS,MZ2,Z2}.

\noindent In a proper vertex coloring $c$ of $G$, a vertex is called nice vertex if for each $j\not= c(u)$, $u$ has a neighbor which is b-vertex of color $j$. A coloring $c$ is called b$^{\ast}$-coloring if there exists a nice vertex in $(G,c)$. We call it ${\rm b}^{\ast}$-coloring because the subgraph induced on b-vertices contains a star graph $K_{1,k-1}$ as subgraph, where the nice vertex is at the center. The ${\rm b}^{\ast}$-chromatic number ${\rm b}^{\ast}(G)$ of $G$ is the maximum number of colors used in a ${\rm b}^{\ast}$-coloring of $G$. It was proved in \cite{Z1,Z2} that there exists an $\mathcal{O}(nm)$ algorithm such that given a graph $G$ on $n$ vertices and $m$ edges, the algorithm provides a ${\rm b}^{\ast}$-coloring in $G$. Note that unlike the b-chromatic number, ${\rm b}^{\ast}(G)=\max {\rm b}^{\ast}(H)$, where the maximum is taken over all connected components $H$ of $G$. For a graph $G$, define $G\vee K_1$ as a graph obtained by adding a vertex $v$ to $G$ and joining $v$ to all the vertices of $G$. It was proved in \cite{Z2} that ${\rm b}(G)={\rm b}^{\ast}(G\vee K_1)-1$.

\noindent In this paper we concentrate on graphs with bounded independence number say $t$. To compute $\chi(G)$ for graph $G$ satisfying $\alpha(G)=2$ is a polynomial time problem since it can be transformed to the maximum matching problem and Edmonds blossom algorithm. For the following reason we expect that to determine the chromatic number of graphs $G$ with $\alpha(G)=t\geq 3$ is an $\NP$-complete problem. The problem PARTITION into TRIANGLES asks is it possible to partition the vertex set of a given graph into triangles. Partition into triangles is a famous $\NP$-complete problem \cite{GJ}. It follows that to determine the minimum partition of $G$ into cliques of size at most three is $\NP$-complete. To determine $\chi(G)$ with $\alpha(G)=3$ is a problem very close to the latter problem.

\noindent It was proved in \cite{KM} that ${\rm b}(G)\leq |V(G)|+1-\alpha(G)$. Although this bound is sharp but it is not useful when $n$ is large with respect to $\alpha(G)$. More results in the literature provide upper bounds for ${\rm b}(G)$ in terms of $\alpha(G)$ and $\chi(G)$. The special case is complement of bipartite graphs for which $\alpha=2$. It was proved in \cite{KZ} that ${\rm b}(G)\leq 4\omega(G)/3$. This bound is generalized in \cite{AK} to the class of graphs with independence number 2. Let $G$ be a graph on $n$ vertices with $\alpha(G)=2$. Then ${\rm b}(G) \leq (n+2\chi(G))/3$. Note that the latter bound implies ${\rm b}(G) \leq 4\chi(G)/3$, since $n\leq \alpha(G)\chi(G)=2\chi(G)$ for these graphs. We extend this bound for an arbitrary $t>2$ in Corollary \ref{main-b}. Another bound in \cite{KZ} is in terms of the clique partition number $\theta(G)$. For a graph $G$, denote the minimum number of vertex disjoint cliques which partition $V(G)$ by $\theta(G)$. It was proved in \cite{KZ} that ${\rm b}(G) \leq \big[\theta(G)^2/(2\theta(G)-1)\big]\chi(G)$. The latter bound is slightly refined to $(2\theta(G)-1){\rm b}(G) \leq [(\theta(G)-1)n+\theta(G)\chi(G))]$ in \cite{AK}. We prove in Corollary \ref{main-b} that $(2\alpha(G)-1){\rm b}(G) \leq [(\alpha(G)-1)n+\alpha(G)\chi(G))]$. Note that $\alpha(G)\leq \theta(G)$. We claim that the bound of Corollary \ref{main-b} significantly improves the related bounds in \cite{AK,KZ}. First note that $\theta(G)-\alpha(G)$ can be arbitrarily large. For this purpose, recall that for each integer $k\geq 2$, there exists a triangle-free graph $G$ such that $\chi(G)=k$ \cite{B}. We have $\alpha(\overline{G})=2$ and $\theta(\overline{G})=k$, as desired. We have the following stronger fact.

\noindent {\bf Fact:} Let $G$ be an Erd\H{o}s-R\'{e}nyi random graph on $n$ vertices sampled from $G(n,1/2)$. Then, with high probability, $\theta(G)-\alpha(G)\rightarrow \infty$ as $n\rightarrow \infty$.

\noindent To prove the fact, note that for any graph $H$, $\alpha(H)=\omega(\overline{H})$ and $\theta(H)=\chi(\overline{H})$. Also note that the complement of a random graph $G(n,1/2)$ is another random graph of type $G(n,1/2)$. Hence, it's enough to prove that $\chi(G)-\omega(G)\rightarrow \infty$ as $n\rightarrow \infty$. By a known result (e.g. \cite{B}) concerning the clique number of random graphs $G=G(n,1/2)$, with high probability $\omega(G)< 2\log n$ as $n\rightarrow \infty$. It follows that with high probability $\chi(G)-\omega(G)\geq n/(\alpha(G))-2\log n > n/(2 \log n) -2\log n$, as $n\rightarrow \infty$. Clearly, $(n/(2\log n)) - (2\log n)\rightarrow \infty$ as $n\rightarrow \infty$. This proves the fact.

\noindent It follows that the bound of Corollary \ref{main-b} on ${\rm b}(G)$ in terms of $\alpha(G)$ significantly improves a similar bound on ${\rm b}(G)$ in terms of $\theta(G)$. For the Grundy number, it was proved in \cite{TWHZ} that $\Gamma(G)\leq (n+\chi(G))/2$. The latter bound implies $\Gamma(G)\leq \chi(G)(\alpha(G)+1)/2$. This bound was also proved in \cite{CKM}. It was proved in \cite{CKM} that the bound for the Grundy number is best possible since there are graphs for which equality holds in the bound. The bound of Theorem \ref{main-b} provides a ratio for ${\rm b}(G)/\chi(G)$ better than $(\alpha(G)+1)/2$, a corresponding bound for $\Gamma(G)/\chi(G)$.

\noindent {\bf The outline of the results:} In Theorem \ref{main} we prove ${\rm b}^{\ast}(G)\leq ((t-1)n+t\chi(G))/(2t-1)$, where $\alpha(G)\leq t$. Corollary \ref{main-b} asserts that the same inequality holds for ${\rm b}(G)$. Proposition \ref{exist} proves that the latter bound is sharp for all $t=\alpha(G)$ and $c=\chi(G)$. These bounds are refined in Propositions \ref{star-bound} and \ref{star-bound-b}. Finally, in Proposition \ref{star-free-b*} a bound for the ${\rm b}^{\ast}$-chromatic number of $K_{1,s}$-free graphs is presented which is much better than a known sharp bound for the b-chromatic number of these graphs.

\section{Bounds in terms of the independence number}

\noindent Let $G$ be a graph with ${\rm b}^{\ast}(G)=b$ and $\chi(G)=\chi$. Let $c$ (resp. $c^{\ast}$) be a proper vertex coloring (resp. ${\rm b}^{\ast}$-coloring) of $G$ using $\chi$ (resp. b) colors. For $i\in \{1, \ldots, \chi\}$ and $j\in \{1, \ldots, b\}$, let $C_i$ (resp. $C^{\ast}_j$) consist of vertices whose color in $c$ (resp. $c^{\ast}$) is $i$ (resp. $j$). Define $C_{ij}=C_i \cap C^{\ast}_j$. By the overlapping array of $c$ and $c^{\ast}$ we mean a $\chi \times b$ array whose $(i,j)$ entry is a set consisting of vertices $u$ such that $c(u)=i$ and $c^{\ast}(u)=j$.

\begin{defin}
We say an overlapping array associated to $(c,c^{\ast})$ has echelon form if there exists a nice vertex $u$ of color $b={\rm b}^{\ast}(G)$ in $c^{\ast}$ with $c(u)=\chi(G)$ and b-vertex neighbors $u_1, \ldots, u_{b-1}$ of $u$ (called as b-neighborhood of $u$) such that $c^{\ast}(u_i)=i$ for each $i\in [b-1]$ and such that $\{u_1, \ldots, u_{b-1}\}$ can be partitioned as
\begin{center}
$\underbrace{u_1, \ldots, u_{i_1}}_\text{$\in C_{\chi-k}$}, \underbrace{u_{i_1+1}, \ldots, u_{i_2}}_\text{$\in C_{\chi-k+1}$}, \ldots, \underbrace{u_{i_{t-1}+1}, \ldots, u_{i_t}}_\text{$\in C_{t-k+\chi-1}$}, \ldots, \underbrace{u_{i_{k-1}+1}, \ldots, u_{i_k}}_\text{$\in C_{\chi -1}$}$,
\end{center}
\noindent where $i_k=b-1$, $i_0=0$ and for each $t\in [k]$, $\{u_{i_{t-1}+1}, \ldots, u_{i_t}\} \subseteq C_{t-k+\chi-1}$.\label{def}
\end{defin}

\begin{figure}[h]
\setlength{\tabcolsep}{6pt}
\begin{center}
\begin{tabular}{|c|c|c|c|c|c|c|c|c|c|c|c|c|}
   \hline
   & $C^{\ast}_1$ &  $C^{\ast}_2$ &  $C^{\ast}_3$ & $C^{\ast}_4$ & $C^{\ast}_5$ & $C^{\ast}_6$ & $C^{\ast}_7$ & $C^{\ast}_8$ & $C^{\ast}_9$ & $C^{\ast}_{10}$ & $C^{\ast}_{11}$ & $C^{\ast}_{12}$\\[0.25eM]
  \hline
   $C_7$ & & &  &  &  &  &  &  & & & & $12$ \\[0.25eM]
  \hline
   $C_6$ & & &  &  &  & & &  & & & $11$ & \\[0.25eM]
  \hline
   $C_5$ & &  & & & & & & & $9$ & $10$ &  &\\[0.25eM]
  \hline
  $C_4$ & & & & & & $6$ & $7$ & $8$ &   &  &  &\\[0.25eM]
  \hline
  $C_3$ & & &  & & $5$ & & & &  &  &    &\\[0.25eM]
  \hline
  $C_2$ & & & $3$ & $4$ &  &  &  &  &  & & &\\[0.25eM]
  \hline
  $C_1$ & $1$ & $2$ & & &  &  &  &  &  & & &\\[0.25eM]
  \hline
  \end{tabular}\\
  \end{center}
\caption{A $7\times 12$ overlapping array of echelon form}\label{echelon}
\end{figure}

\noindent Note that in Definition \ref{def}, $C_{\chi-k}$ and $C_1$ are not necessarily equal. Figure \ref{echelon} depicts an overlapping array of echelon form of parameters $\chi=7$ and $b=12$. A number say $p$ in an entry $(i,j)$ of the array means that a b-vertex of color $p$ belongs to $C_i\cap C^{\ast}_j$.

\begin{prop}
Let $G$ be a graph, $\chi(G)=k$ and ${\rm b}^{\ast}(G)=b$. Then $G$ admits a $k\times b$ overlapping array of echelon form.
\end{prop}

\noindent \begin{proof}
Let $c^{\ast}$ be a maximum b$^{\ast}$-coloring of $G$ with $b$ colors and $u$ be a nice vertex in $c^{\ast}$. Let $u_i$ be a neighbor of $u$ such that $u_i$ is b-vertex and $c^{\ast}(u_i)=i$, for each $i\in [b-1]$. Set $B=\{u_1, \ldots, u_{b-1}\}$. Let $C_1^{\ast}, \ldots, C_b^{\ast}$ be the color classes in $c^{\ast}$. We begin with stating a general fact concerning $c^{\ast}$. Let $\sigma: \{1, \ldots, b\} \rightarrow \{1, \ldots, b\}$ be a permutation (i.e. bijection) such that $\sigma(b)=b$. Define a new coloring $\sigma(c^{\ast})$ defined as follows. For each $j\in [b]$, class of vertices of color $j$ in $\sigma(c^{\ast})$ is $C_i^{\ast}$, where $\sigma(i)=j$. Then $\sigma(c^{\ast})$ is a valid b$^{\ast}$-coloring of $G$ with $b$ colors. The reason is that $u$ is still a nice vertex in $\sigma(c^{\ast})$ of color $b$.

\noindent There exists a $\chi(G)$ coloring $c$ of $G$ such that $c(u)=k$. Let $A(c,c^{\ast})$ be the overlapping array associated to $(c,c^{\ast})$. Note that $c(u_i)\not= k$, for each $i$. We start with the pair $(c,c^{\ast})$. Let $r\in [k-1]$ be such that $c(u_{b-1})=r$. Let $u_{j_1}, \ldots, u_{j_p}$ be other (if any) b-vertices such that $c(u_{j_i})=r$ ($j_i\not= b-1$, for each $i$). Let $\sigma$ be a suitable permutation of $1, \ldots, b-2$ such that for each $i\in [p]$, the color of vertex $u_{j_i}$ in $\sigma(c^{\ast})$, belongs to $\{b-p-1, \ldots, b-2\}$. For simplicity denote $\sigma(c^{\ast})=c_1^{\ast}$. If $r=k-1$ then $u_{b-1}$ is in position $(k-1,b-1)$ of $A(c,c^{\ast})$ and we do nothing concerning $u_{b-1}$. Otherwise, let $c_1$ be a $\chi(G)$-coloring obtained by exchanging colors $k-1$ and $r$ in $c$. We have $c_1(u_{b-1})=k-1$. In either case we have the following property concerning the new pair $(c_1,c_1^{\ast})$. The vertices of $B$ (b-vertex neighbors of $u$) which are in row $k-1$ of the new array $A(c_1,c_1^{\ast})$, are placed in columns $b-p-1, \ldots, b-1$ and only $u_{b-1}$ is in position $(k-1, b-1)$. The argument includes the possible case $p=0$.

\noindent Now in $A(c_1,c_1^{\ast})$ we ignore rows $k, k-1$ and columns $b-p-1, \ldots, b-1, b$. Repeat a same technique in the remaining subarray and obtain a new array $A(c_2,c_2^{\ast})$ satisfying the property that vertices of $B$ which are in row $k-2$ of $A(c_2,c_2^{\ast})$, are placed in columns $b-q, \ldots, b-p-2$, for some $q\geq p+2$. By repeating this scenario we finally obtain a final pair $(c_f,c_f^{\ast})$ such that $A(c_f,c_f^{\ast})$ has the following property. For each $r,s\in \{1, \ldots, b-1\}$ with $r>s$ if $c_f(u_r)=p$ and $c_f(u_s)=q$ then $p\geq q$. The resulting array has echelon form, as desired.
\end{proof}

\noindent Consider a pair $(c,c^{\ast})$ for a graph $G$ such that its corresponding array has echelon form, where $c$ (resp. $c^{\ast}$) is a $\chi(G)$-coloring (resp. b$^{\ast}$-coloring of $G$ with b$^{\ast}(G)$ colors). Let $B$ be a set of b-neighbors for a nice vertex $u$ in $c^{\ast}$. Let the color classes in $c$ be $C_1, \ldots, C_k$, where $k=\chi(G)$. Clearly, $C_k\cap (B\cup \{u\})=\{u\}$. We say the array has a modified form if $0\leq |C_1\cap B| \leq |C_2\cap B| \leq \ldots \leq |C_{\chi(G)-1}\cap B| \leq t$ and for some $\lambda\in \{2, \ldots, t-1\}$, $\big\{|C_i\cap B|: i=1, 2, \ldots, \chi(G)\big\}\subseteq \{0,1,\lambda,t\}$, where the multiplicity of $\lambda$ is at most $1$. The echelon array in Figure \ref{modif} has modified form.

\begin{figure}[h]
\setlength{\tabcolsep}{6pt}
\begin{center}
\begin{tabular}{|c|c|c|c|c|c|c|c|c|c|c|c|c|}
   \hline
   & $C^{\ast}_1$ &  $C^{\ast}_2$ &  $C^{\ast}_3$ & $C^{\ast}_4$ & $C^{\ast}_5$ & $C^{\ast}_6$ & $C^{\ast}_7$ & $C^{\ast}_8$ & $C^{\ast}_9$ & $C^{\ast}_{10}$ & $C^{\ast}_{11}$ & $C^{\ast}_{12}$\\[0.25eM]
  \hline
   $C_7$ & & &  &  &  &  &  &  & & & & $12$ \\[0.25eM]
  \hline
   $C_6$ & & &  &  &  & & &  & $9$ & $10$ & $11$ & \\[0.25eM]
  \hline
   $C_5$ & &  & & & & $6$ & $7$ & $8$ & & &  &\\[0.25eM]
  \hline
  $C_4$ & & & & $4$ & $5$ & & &  &  &  &  &\\[0.25eM]
  \hline
  $C_3$ & & &  $3$ & & &  &  &  &  & & &\\[0.25eM]
  \hline
  $C_2$ & & $2$ & & &  &  &  &  &  & & &\\[0.25eM]
  \hline
  $C_1$ & $1$ & & & &  &  &  &  &  & & &\\[0.25eM]
  \hline
  \end{tabular}\\
  \end{center}
\caption{Echelon array of Figure \ref{echelon} transformed into a modified echelon form}\label{modif}
\end{figure}

\begin{prop}
Let $G$ be a graph with $\alpha(G)\leq t$. Let $c$ consisting of color classes $C_1, \ldots, C_{\chi(G)}$ (resp. $c^{\ast}$) be a $\chi(G)$-coloring (resp. a b$^{\ast}$-coloring using b$^{\ast}(G)$ colors) of $G$ such that their overlapping array has echelon form. Let $u$ be a nice vertex and $B$ be a b-neighborhood set of $u$ in $c^{\ast}$. Let $g_i=|\{C_j:|C_j\cap B|=i\}$, for $i\in [t]$. Then $|V(G)|\geq |B|+{\sum}_i ig_i$. Moreover, the minimum value for ${\sum}_i ig_i$ is occurred when $\{g_i|1\leq i\leq t\}\subseteq \{1, p, t\}$, for some $p\in \{2, \ldots, t-1\}$. That is ${\sum}_i ig_i$ is occurred when the array has modified echelon form.\label{lower}
\end{prop}

\noindent \begin{proof}
Fix an $i\in \{1, \ldots, t\}$. Let $C$ be a class in $c$ such that $|B\cap C|=i$. Assume that $B\cap C=\{u_{k_1}, \ldots, u_{k_i}\}$. Then for each two b-vertices $u_{k_r}$ and $u_{k_s}$, $1\leq r\not= s \leq i$, $u_{k_r}$ needs a neighbor in $c^{\ast}$ of color $s$ and vice versa. Since  $u_{k_1}, \ldots, u_{k_i}$ are independent then these neighbors are distinct vertices from $B$. It follows that $n\geq |B|+{\sum}_i ig_i$.

\noindent We prove the second part. Denote $\Psi(g_1, \ldots, g_t)={\sum}_i ig_i$ and $\delta(g_1, \ldots, g_t)={\sum}_{i=1}^{\chi(G)-1} g_i(\chi(G)-i)$. Let $i$ be a smallest index with $g_i\geq 2$. There are two possibilities.

\noindent Case 1. $g_{i+1}=1$.

\noindent In this case we switch two classes $C_i$ and $C_{i+1}$. By this switch the b$^{\ast}$-coloring is not changed but for the new parameters $g^{new}_i$ and $g^{new}_{i+1}$ we have $g^{new}_i=1$ and $g^{new}_{i+1}\geq 2$. Then we continue with $g^{new}_{i+1}$.

\noindent Case 2. $2\leq g_{i+1} \leq t-1$.

\noindent In this case we move $\max \{g_i-1,t-g_{i+1}\}$ vertices from $C_i \cap B$ to $C_{i+1} \cap B$. If $g^{new}_i=1$ then obviously $\Psi^{new}< \Psi$. But if $g^{new}_i\geq 2$ then $\Psi^{new}= \Psi$. But $\delta(g_1^{new}, \ldots, g_{\chi(G)-1}^{new})<\delta(g_1, \ldots, g_{\chi(G)-1})$.

\noindent By continuing the above procedure we eventually obtain a distribution of vertices in $B$ such as  $g^f_1, \ldots, g^f_t$ such that the corresponding quantities $\Psi^f$ and $\delta^f$ cannot be decreased. It follows that there exists no $i$ such that $g^f_i\geq 2$ and $g^f_{i+1} \leq t-1$. This only happens when there exists $p\in \{2, \ldots, t-1\}$ such that for each $i$, $g^f_i\in \{1, p, t\}$; and the multiplicity of $p$ is at most one.
\end{proof}

\begin{thm}
Let $G$ be a graph on $n$ vertices and $t=\alpha(G)$. Then
\begin{center}
$(2t-1){\rm b}^{\ast}(G)\leq (t-1)n+t\chi(G).$
\end{center}\label{main}
\end{thm}

\noindent \begin{proof}
Let $c$ (resp. $c^{\ast}$) be a $\chi(G)$-coloring (resp. a b$^{\ast}$-coloring using b$^{\ast}(G)$ colors) of $G$ such that their overlapping array has echelon form. Let $B$ be a b-neighborhood set in $c^{\ast}$ and suppose that there are $g'_i$ classes $C$ in $c$ such that $|C\cap B|=i$, for each $i\in \{1, \ldots, t\}$. By Proposition \ref{lower}, $n\geq |B|+\Psi(g'_1, \ldots, g'_t)\geq |B|+\Psi(g_1, \ldots, g_t)$, where for each $i$, $g_i\in \{1, p, t\}$, for some $p$. Suppose that the multiplicity of $t$ (resp. $1$) is $g_t$ (resp. $k$). Suppose also that first $q$ classes in $c$ does not intersect $B$. It follows that b$^{\ast}(G)=tg+p+k$ and $\chi(G)=1+g+k+q$. Fix $i\geq 2$, and consider the b-vertices in $C_i\cap B$. Let $u_r$ and $u_s$ be arbitrary b-vertices in $C_i$. Then $u_r$ needs a neighbor of color $s$ and vice versa. This introduces $|C_i\cap B|$ extra vertices (denoted as $\nu$) for the graph, since $u_r$ and $u_s$ are independent. It follows that
\begin{center}
$n\geq tg+p+k+q+\nu \geq tg+p+k+q+tg+p\geq 2tg+2p+k+q.$
\end{center}
\noindent Let $\zeta\geq 0$  be such that $n=2tg+2p+k+q+\zeta$. Note that $q$ only appears positively in the right hand side (RHS) of $(2t-1){\rm b}^{\ast}(G)\leq (t-1)n+t\chi(G)$. So it's enough to prove the bound for the case $q=0$.
\noindent The inequality $(2t-1){\rm b}^{\ast}\leq (t-1)n+t\chi(G)$ is equivalent to
\begin{center}
$(2t-1)(tg+p+k)\leq (t-1)(2tg+2p+k+\zeta)+t(1+g+k).$
\end{center}
\noindent For LHS and RHS of the inequality we have
\begin{center}
$LHS=2t^2g+2tp+2tk-tg-p-k,$
\end{center}
\begin{center}
$RHS=2t^2g+2tp+2tk+t\zeta-tg-2p-k-\zeta+t.$
\end{center}
\noindent It follows that $RHS-LHS=(t-1)\zeta+t-p\geq 0$. This completes the proof. As a bypass result we obtain that if quality holds in the bound then $q=\zeta=0$.
\end{proof}

\begin{cor}
For any graph $G$ on $n$ vertices,
$$(2\alpha(G)-1){\rm b}(G)\leq (\alpha(G)-1)n+\alpha(G)\chi(G).$$\label{main-b}
\end{cor}

\noindent \begin{proof}
We have ${\rm b}(G)={\rm b}^{\ast}(G\vee K_1)-1$. Then $(2t-1){\rm b}(G)=(2t-1)\big({\rm b}^{\ast}(G\vee K_1)-1\big) \leq (t-1)(n+1)+t(\chi(G)+1)-(2t-1)\leq (t-1)n+t\chi(G)$.
\end{proof}

\noindent The bound of Corollary \ref{main-b} is sharp for all $t\geq 2$ and $\chi\geq 3$.

\begin{prop}
Let $t\geq 2$ and $c\geq 3$ be any integers. Then there exists a graph $G$ such that $\alpha(G)=t$, $\chi(G)=c$ and ${\rm b}(G)=\big\lfloor \big((t-1)|V(G)|+t\chi(G)\big)/(2t-1) \big\rfloor$.\label{exist}
\end{prop}

\noindent \begin{proof}
We construct a graph $G$ such that $\chi(G)=c$ and ${\rm b}^{\ast}(G)=\lfloor \big(t^2/(2t-1)\big)\chi(G) \rfloor$. Based on the proof of Proposition \ref{lower}, we need to find $g$, $p$ and $k$ such that $\chi(G)=c=k+g+1$ (or $c=k+g$, for the case $p=0$) and ${\rm b}^{\ast}(G)=gt+p+k$ (or ${\rm b}^{\ast}(G)=gt+k$, for the case $p=0$). Write for simplicity  $\beta=\lfloor \big(t^2/(2t-1)\big) c \rfloor$. In order to have enough room for placing necessary and sufficient neighbors of the b-vertices, the number of available vertices i.e. $k(t-1)+t-p$ (or $k(t-1)$, if $p=0$) should be at least the number of required vertices i.e. $tg+p$ ($tg$, if $p=0$). Then the existence problem bifurcates as follows.

\noindent Case $p=0$: In this case we should have $k(t-1)\geq tg$ and also $c=k+g$ and $\beta=gt+k$. It follows that the problem reduces to find $k$ such that $1\leq k \leq c$ and

\begin{center}$\begin{cases}
\beta=tc-k(t-1)\quad\hspace{3cm} \\
			
tc\leq k(2t-1) \quad\hspace{1.75cm}.
\end{cases}$
\end{center}

\noindent Case $p\geq 1$: In this case we should have $k(t-1)+t-p\geq tg+p$ or $k(t-1)\geq t(g-1)+2p-t$ and also $c=k+g+1$ and $\beta=gt+p+k$. It follows that the problem reduces to find $k$ and $p$ such that $1\leq k \leq c$, $2\leq p \leq t-2$ and

\begin{equation}
\begin{cases}
\beta = t(c - 1 - k) + k + p \quad\hspace{3cm} \\
			
2p+t(c-1) \leq k(2t-1)+t \quad\hspace{1.75cm}.
\end{cases}
\end{equation}

\noindent Note that $tc-\beta\geq t-1$. By dividing $tc-\beta$ to $t-1$, let $k$ be a positive integer such that $tc-\beta=k(t-1)+r$, for some $0\leq r \leq t-2$.

\noindent Case 1: $r = 0$

\noindent This case corresponds to $p=0$. In this case the equality $tc-k(t-1)=\beta$ in ${\bf (1)}$ holds. We prove $tc \geq k(2t-1)$. For this purpose note that $\beta(2t-1)\leq ct^2$. Then $(tc-k(t-1))(2t-1)\leq ct^2$ and $2t^2c-tc-k(t-1)(2t-1)\leq ct^2$. The latter inequality implies $t^2c-tc\leq k(t-1)(2t-1)$ and $tc(t-1)\leq k(t-1)(2t-1)$. Now, since $t-1\geq 1$ then $tc\geq k(2t-1)$, as desired. The conditions in ${\bf (1)}$ hold.

\noindent We show that $k$ is a valid solution by proving that $1\leq k \leq c$. Note that $tc\geq 6$ then $k(2t-1)\geq 6$. Hence, $k\geq 1$. To prove $k\leq c$ we start from $ct^2< (\beta+1)(2t-1)$ or equivalently $ct^2< 2t^2c-tc-k(t-1)(2t-1)+2t-1$. It yields $k(2t-1)<tc+2+(1/(t-1))$. If we assume $k\geq c+1$, then $(c+1)(2t-1)\leq tc+2+(1/(t-1))$. It simplifies to $c(t-1)+2t-3\leq 1/(t-1)$. The left side is at least $4$, a contradiction.

\noindent \noindent Case 2: $r \geq 2$

\noindent In this case define $p = t - r$. Since $1 \leq r \leq t-2$, then $2\leq p \leq t-1$, which satisfies the required bounds for $p$. For $p$ we should have $\beta = t(c - 1 - k) + k + p$. To prove, substitute $r=t-p$ into the division equation $tc-\beta=k(t-1)+r$ and obtain $tc-\beta=k(t-1)+t-p$. It follows that $\beta=(c - 1 - k) + k + p$, as desired.

\noindent Now we prove the inequality $k(2t-1)+t\geq 2p+t(c-1)$ in ${\bf (2)}$. Substitute $\beta=t(c-1)-k(t-1)$ into the lower bound $\beta(2t-1)\leq ct^2$ and obtain $\big(tc-t-k(t-1)+p\big)
(2t-1)\leq ct^2$. It implies the following inequalities
\begin{center}
$t^2c-tc-2t^2+t+2pt-p\leq k(t-1)(2t-1)$
\end{center}
\begin{center}
$\Rightarrow tc(t-1)-2t(t-1)-t+2p(t-1)+p\leq k(t-1)(2t-1)$
\end{center}
\begin{center}
$\Rightarrow (tc-2t+2p)(t-1)+p-t\leq k(t-1)(2t-1)$
\end{center}
\begin{center}
$\Rightarrow tc-2t+2p+(p-t)/(t-1)\leq k(2t-1).$
\end{center}
\noindent Note that $-1<(p-t)/(t-1)<0$ and both $tc-2t+2p$ and $k(2t-1)$ are integers. Then $tc-2t+2p\leq k(2t-1)$ and $k(2t-1)+t\geq 2p+t(c-1)$, as desired.

\noindent A proof similar to the one in case $r=0$ proves $k\leq  c$. We omit the details.


\begin{figure}[h]
\setlength{\tabcolsep}{6pt}
\begin{center}
\begin{tabular}{|c|c|c|c|c|c|c|c|c|c|c|c|c|}
   \hline
   & $C^{\ast}_1$ &  $C^{\ast}_2$ &  $C^{\ast}_3$ & $C^{\ast}_4$ & $C^{\ast}_5$ & $C^{\ast}_6$ & $C^{\ast}_7$ & $C^{\ast}_8$ & $C^{\ast}_9$ & $C^{\ast}_{10}$ & $C^{\ast}_{11}$ & $C^{\ast}_{12}$\\[0.25eM]
  \hline
   $C_7$ & & &  &  &  & $6$ &  &  & $9$ & & & \textcolor{red}{$12$} \\[0.25eM]
  \hline
   $C_6$ & & &  &  &  & & &  & \textcolor{red}{$9$} & \textcolor{red}{$10$} & \textcolor{red}{$11$} & \\[0.25eM]
  \hline
   $C_5$ & &  & & & & \textcolor{red}{$6$} & \textcolor{red}{$7$} & \textcolor{red}{$8$} & & &  &\\[0.25eM]
  \hline
  $C_4$ & & & & \textcolor{red}{$4$} & \textcolor{red}{$5$} & & &  &  &  &  &\\[0.25eM]
  \hline
  $C_3$ & & &  \textcolor{red}{$3$} & & &  &  & & & $10$ & $11$ &\\[0.25eM]
  \hline
  $C_2$ & & \textcolor{red}{$2$} & & & & & $7$ & $8$ &   & & &\\[0.25eM]
  \hline
  $C_1$ & \textcolor{red}{$1$} & & & $4$ & $5$ &  &  &  &  & & &\\[0.25eM]
  \hline
  \end{tabular}\\
  \end{center}
\caption{A graph realizing Proposition \ref{exist} with $|V(G)|=20$, $\alpha=3$, $\chi=7$, ${\rm b}^{\ast}=12$ for which equality holds in Corollary \ref{main-b} (with parameters $g=p=2$ and $k=4$)}\label{full}
\end{figure}

\noindent In both cases, we proved the existence of $g$, $k$ and $p$ satisfying the conditions ${\bf (1)}$ and ${\bf (2)}$. The rest of the construction is as follows. We hold a $c \times \beta$ array and put a vertex $u$ in position $(c,\beta)$. Vertex $u$ will be a nice vertex of color $\beta$. We first specify position of $gt$ many b-vertices of colors $gt+k, gt+k-1, \ldots, k+1$. Place $t$ vertices of colors $\beta-jt, \beta-jt+1, \ldots, \beta-(j-1)t-1$ in row $c-j$ of the array for each $j\in \{1, \ldots, g\}$, in such a way that the array has a modified echelon form. We should specify position of $k$ many b-vertices. One of them is the nice vertex $u$ placed in the position $(c,\beta)$. For each $i\in \{1, \ldots, k-1\}$, place a vertex of color $i$ in position $(i,i)$ of the table. In case that $p=0$ we need no more vertices. But if $p>1$ we introduce the remaining $p=\beta-gt-k$ vertices of colors $k, k+1, \ldots, k+p-1$ as b-vertices all in the row $k$ corresponding to the class $C_k$. Since the inequality $k(t-1)+t-p\geq tg+p$ hold then we have enough room to add extra vertices to the array and add necessary edges between the new and old vertices to make the previous vertices b-vertex. We obtain a clique of size $k$. See Figure \ref{full} in which red entries are b-vertices and the related parameters are $t=3$, $c=7$, $n=20$, $g=2$, $k=4$, $p=2$ and $\beta=12$. Call the resulting graph $G$ in which vertex $u$ is nice in the underlying ${\rm b}^{\ast}$-coloring. We have $|V(G)|=2gt+p+k$ and $\chi(G)=\omega(G)=c$. The required equality for ${\rm b}(G)$ is obtained as follows
\begin{center}
${\rm b}^{\ast}(G) = \big\lfloor \big[(t-1)|V(G)|+t\chi(G)\big]/(2t-1) \big\rfloor \leq {\rm b}(G)\leq \big\lfloor \big[(t-1)|V(G)|+t\chi(G)\big]/(2t-1) \big\rfloor.$
\end{center}
\end{proof}

\section{More bounds involving independent sets}

\noindent Let $G$ be a graph and $u$ a vertex of $G$. We say $u$ has $t$-star-degree at least $d$ if there exists a subset $D\subseteq N(u)$ such that $|D|=dt$ and $G[D]$ can be partitioned into $d$ vertex disjoint independent subsets each of size $i$. Define the $t$-star-degree of $u$ as the maximum $d$ satisfying the latter properties and denote it as $d_t^{\sf star}(u)$. Define the maximum $t$-star-degree of $G$ as ${\Delta}_t^{\sf star}(G)={\max}_{u} d_t^{\sf star}(u)$. Refining the bound of Theorem \ref{main}, the following bound is in terms of $|V(G)|$, $\chi(G)$, $\alpha(G)$, ${\rm b}^{\ast}(G)$ and the new quantity ${\Delta}_t^{\sf star}(G)$. It follows that for graphs with limited ${\Delta}_t^{\sf star}(G)$, the new upper bound improves significantly the bound of Theorem \ref{main}. An example is family of graphs with $\alpha(G)=t$ but ${\Delta}_t^{\sf star}(G)\leq 1$. The question whether ${\Delta}_t^{\sf star}(G)=k$ ($k$ fixed) can be decided in time $\mathcal{O}(n^{tk})$ as argued later.

\begin{prop}
Let $G$ be a graph on $n$ vertices and ${\Delta}_t^{\sf star}(G)=g$, where $t\geq 3$. Then
\begin{center}
$(2t-3){\rm b}^{\ast}(G)\leq (t-2)n+(t-1)\chi(G)+g(t-2).$
\end{center}
\label{star-bound}
\end{prop}

\noindent \begin{proof}
The proof idea is similar to that of Theorem \ref{main}. Note that with the assumptions of the proposition the related quantity $\Psi(g_1, \ldots, g_t)$ corresponding to a pair $(c,c^{\ast})$ and an arrangement of b-vertices $B$ has a lower bound when the corresponding arrangement has modified echelon form with the restriction that the number of classes $C_i$ satisfying $|C_i\cap B|=t$ is at most $g$. This restriction implies that $\{|C_i\cap B|: 1\leq i \leq \chi(G)\}=\{1, p, t-1, t\}$, for some $2\leq p\leq t-2$. The case that $p$ does not exist is also possible. Note that if $t=3$ then $\{|C_i\cap B|: 1\leq i \leq \chi(G)\}=\{1, 2, 3\}$ and $p$ does not exist. Define $|\{i: |C_i\cap B|=t-1\}|=f$, $|\{i: |C_i\cap B|=t\}|=g$ and $|\{i: |C_i\cap B|=1\}|=k$. Let also $|\{i: |C_i\cap B|=0\}=q$. Now we can expand the main parameters in terms of $f, g, p, k$ and $q$. We have
\begin{center}$
\begin{cases}
\chi(G)=f+g+k+1+q \quad\hspace{3cm} \\
			
{\rm b}^{\ast}(G)=(t-1)f+tg+p+k+1 \quad\hspace{0.3cm} \\
			
n\geq 2(t-1)f+2tg+2p+k+q \quad\hspace{1.75cm}.
\end{cases}$
\end{center}

\noindent The assertion is $(2t-3){\rm b}^{\ast}(G)\leq (t-2)n+(t-1)\chi(G)+g(t-2)$. Note that $q$ only appears positively in RHS. So it's enough to prove the bound for the case $q=0$. Write for simplicity $\chi(G)=c$ and ${\rm b}^{\ast}(G)=b$. We have $c=k+1+f+g$, $b=k+p+(t-1)f+tg$ and $n\geq 2(t-1)f+2tg+2p+k$. Define a variable $\zeta\geq 0$ so that $n = k + 2p + 2(t-1)f + 2gt + \zeta$. It follows that
\begin{center}
$RHS= (t-2)[k + 2p + 2(t-1)f + 2gt + \zeta] + (t-1)[k + f + 1 + g] + g(t-2)$
\end{center}
\begin{center}
$=(2t-3)k + (2t-4)p + (2t-3)(t-1)f + g(2t^2 - 2t - 3) + (t-2)\zeta + t - 1.$
\end{center}
\noindent The LHS expression is $(2t-3)b = (2t-3)[k + p + (t-1)f + tg]$. It is easily checked that
\begin{center}
$RHS - LHS = (t-2)\zeta - p + g(t-3) + t - 1.$
\end{center}
\noindent Since $(t-2)\zeta$, $t-p-1$ and $g(t-3)$ are non-negative, then $RHS-LHS \geq t-2$. The inequality $(2t-3)b\leq (t-2)n + c(t-1) + g(t-2)$ is proved.
\end{proof}

\noindent In a graph $G$ denote by $\Delta_t^{\sf indep}(G)$ the maximum number of vertex disjoint independent subsets each of size $t$. The new quantity is equivalent to the maximum number of vertex disjoint cliques of size $t$ in the complement of $G$.

\begin{prop}
Let $G$ be a graph on $n$ vertices such that $3\leq \alpha(G)\leq t$. Then
\begin{center}
$(2t-3){\rm b}(G)\leq (t-2)n+(t-1)\chi(G)+(t-2)\Delta_t^{\sf indep}(G).$
\end{center}\label{star-bound-b}
\end{prop}

\noindent \begin{proof}
Note that $\Delta_t^{\sf indep}(G)=\Delta_t^{\sf star}(G\vee K_1)$ and recall that ${\rm b}(G)={\rm b}^{\ast}(G\vee K_1)-1$. Applying Proposition \ref{star-bound} for $G\vee K_1$ we obtain
\begin{center}
$(2t-3){\rm b}(G)=(2t-3){\rm b}^{\ast}(G\vee K_1)-(2t-3)$
\end{center}
\begin{center}
$\leq (t-2)(n+1)+(t-1)(\chi(G)+1)+\Delta_t^{\sf star}(G\vee K_1)(t-2)-(2t-3)$
\end{center}
\begin{center}
$=(t-2)n+(t-1)\chi(G)+(t-2)\Delta_t^{\sf indep}(G).$
\end{center}
\end{proof}

\noindent We make a comment on the computational complexity of $\Delta_t^{\sf star}(G)$ and $\Delta_t^{\sf indep}(G)$. Both problems seem to be $\NP$-hard. Let $k$ be a positive integer. By an exhaustive search we can decide whether $d_t^{\sf star}(u)=k$ by consuming ${\mathcal{O}}(\Delta(G)^{tk})$ search and check operations. Hence, the complexity of $\Delta_t^{\sf indep}(G)$ is $|V(G)|{\mathcal{O}}(\Delta(G)^{tk})$. Since $t$ is fixed, for fixed $k$ the latter complexity is polynomial in $|V(G)|$. A similar argument proves that the complexity of $\Delta_t^{\sf indep}(G)$ is ${\mathcal{O}}(|V(G)|^{tk})$.

\noindent It was proved in \cite{KZ} that if $G$ is a $K_{1,s}$-free graph then b$(G)\leq (s-1)(\chi(G)-1)+1$, where $s\geq 3$. We obtain a much better bound for the ${\rm b}^{\ast}$-chromatic number of these graphs. Note that for $K_{1,2}$-free graphs we have ${\rm b}^{\ast}(G)= \chi(G)$.

\begin{prop}
Let $G$ be a $K_{1,s}$-free graph on $n$ vertices, where $s\geq 2$. Then
\begin{center}
$(2s-3){\rm b}^{\ast}(G)\leq (s-2)n+(s-1)\chi(G).$
\end{center}
\noindent Moreover, the bound is sharp for all $s\geq 2$ and $\chi(G)\geq 2$. \label{star-free-b*}
\end{prop}

\noindent \begin{proof}
As in Proposition \ref{modif}, let $c$ and $c^{\ast}$ be colorings of $G$ such that their array has echelon form with $B$ as the corresponding b-neighborhood for a nice vertex $u$. Similarly, define $g_i=|\{C_j:|C_j\cap B|=i\}|$. Since $G$ is $K_{1,s}$-free then $i\leq s-1$ for each $i$ with $g_i\not= 0$. The situation is like graphs with at most $s-1$ independent vertices. A same lower bound for $n$ as for graphs with the independence number at most $s-1$ holds. The required bound is deduced from Theorem \ref{main}.

\noindent For the sharpness, by Proposition \ref{exist}, there exists a graph $G$ on $n$ vertices satisfying $\alpha(G)=s-1$ such that ${\rm b}^{\ast}(G)=\big\lfloor \big[(s-2)n+(s-1)\chi(G)\big]/(2s-3) \big\rfloor$. Graph $G$ is $K_{1,s}$-free and satisfies the desired property.
\end{proof}

\section{Concluded Remarks}

\noindent We end the paper by mentioning some unexplored research areas. In addition to the open problems concerning the ${\rm b}^{\ast}$-coloring mentioned in \cite{Z2}, it is interesting to explore ${\rm b}^{\ast}$-perfect graphs. That is graphs $G$ satisfying ${\rm b}^{\ast}(H)=\chi(H)$, for each induced subgraph $H$ of $G$. The ${\rm b}$-perfectness has been the research subject of many articles e.g. \cite{HK,HMM}. The ${\rm b}^{\ast}$-chromatic number of $P_4$-tidy graphs is another interesting problem. The related references are \cite{BDMMV,VBK,Z2}.

\noindent Theorem \ref{main-b} implies ${\rm b}(G)\leq (t^2/(2t-1))\chi(G)$. It follows that a polynomial time algorithm obtained by the b-coloring procedure provides a ratio ${\rm b}(G)/\chi(G)\leq t^2/(2t-1)$ which is better than a corresponding ratio $(t+1)/2$ for the First-Fit coloring. Is there any polynomial time vertex coloring algorithm $\mathcal{A}$ such that $\mathcal{A}(G)/\chi(G)$ is strictly better than $t^2/(2t-1)$, where $\mathcal{A}(G)$ is the number of colors used by $\mathcal{A}$ and $t\geq \alpha(G)\geq 3$? We insist on $t\geq 3$ since for graphs with $\alpha(G)=2$ there exits a polynomial time algorithm which computes $\chi(G)$.


\end{document}